\documentclass[11]{amsart}

\usepackage{amsmath,amsfonts,amssymb}
\usepackage{array,delarray}
\usepackage[dvips]{graphicx}
\usepackage{dsfont}
\usepackage{xspace}
\usepackage{color}

\newcommand{\ju}{\wedge}
\newcommand{\Ju}{\bigwedge}
\newcommand{\la}{\lambda}
\newcommand{\te}{\theta}

\newcommand{\s}{\mathcal{S}}
\newcommand{\cc}{\mathcal{C}}
\newcommand{\jj}{\mathcal{J}}

\newcommand{\lcm}{\mathrm{lcm}}

\newcommand{\vv}{\mathbf{v}}
\newcommand{\uu}{\mathbf{u}}

\newcommand{\dd}{\mathfrak{D}}

\newtheorem{theorem}{Theorem}[section]   
\newtheorem{corollary}[theorem]{Corollary}     
\newtheorem{lemma}[theorem]{Lemma}         
\newtheorem{proposition}[theorem]{Proposition}  

\theoremstyle{definition}
\newtheorem{definition}[theorem]{Definition}   

\theoremstyle{remark}
\newtheorem{remark}[theorem]{Remark}        
\newtheorem{example}[theorem]{\emph{Example}}        

\numberwithin{equation}{section}     

\def\ds {\displaystyle}

\usepackage{graphicx}
\usepackage{type1cm}
\usepackage{eso-pic}
\usepackage{color}

\begin{document}

\title[Semilattice Networks]
{The Dynamics of Semilattice Networks}

\author{Alan~Veliz-Cuba and Reinhard Laubenbacher}

\maketitle

\begin{abstract}

Time-discrete dynamical systems on a finite state space have been
used with great success to model natural and engineered systems such
as biological networks, social networks, and engineered control
systems. They have the advantage of being intuitive and models can
be easily simulated on a computer in most cases; however, few
analytical tools beyond simulation are available. The motivation for
this paper is to develop such tools for the analysis of models in
biology. In this paper we have identified a broad class of discrete
dynamical systems with a finite phase space for which one can derive
strong results about their long-term dynamics in terms of properties
of their dependency graphs. We classify completely the limit cycles
of semilattice networks with strongly connected dependency graph and
provide polynomial upper and lower bounds in the general case.

\end{abstract}

\section{Introduction and Background}
Time-discrete dynamical systems on a finite state space play an important
role in several different contexts. Examples of such systems include
Boolean networks, cellular automata, agent-based models, and
finite state machines, to name a few. This modeling paradigm
has been used with great success to model natural and engineered systems
such as biological networks, social networks, and engineered control systems.
It has the advantage of being intuitive and models can be easily simulated on
a computer in most cases. A disadvantage of discrete models of this type
is that few analytical tools beyond simulation are available. The motivation
for this paper is to develop such tools for the analysis of models in
biology, but the results are of independent
mathematical interest. Some theoretical
results have been proven for Boolean networks and are reviewed in
\cite{CBN}. In that paper the authors study conjunctive Boolean networks,
that is, Boolean networks for which the future state of each node is
computed using the Boolean AND operator. It is shown that the dynamics of
such networks is controlled strongly by the topology of the network. The
current paper shows that the results in \cite{CBN} are valid much more
broadly. We briefly describe the main results in \cite{CBN}.

A Boolean network $f$ on $n$ nodes $x_1,\ldots , x_n$ can be viewed as a
time discrete dynamical system over the Boolean field $\mathbf F_2$:
$$
f = (f_1,\ldots, f_n): \mathbf F_2^n\longrightarrow \mathbf F_2^n,
$$
where the \emph{coordinate functions} $f_i$ are the Boolean functions
assigned to the nodes of the network.
We can associate two directed graphs to $f$. The
\emph{dependency graph} (or wiring diagram) $\dd(f)$ of $f$ has $n$
vertices corresponding to the variables $x_1,\ldots ,x_n$ of $f$.
There is a directed edge $i \rightarrow j$ if the function $f_j$
depends on $x_i$ (i.e. there is an instance where changing the value
of $x_i$ changes the value of $f_j$). That is, $\dd(f)$ encodes the
variable dependencies in $f$.
The dynamics of $f$ is encoded by its \emph{phase space}, denoted by
$\s(f)$. It is the directed graph with vertex set $\mathbf F_2^n$ and a
directed edge from $\uu$ to ${\vv}$ if $f(\uu) = \vv$.
Thus, the graph $\dd(f)$ encodes part of the structure of $f$ and the
graph $\s(f)$ encodes its dynamics. The results in \cite{CBN} relate
these two graphs, deriving information about $\s(f)$ from $\dd(f)$, in
the case where each $f_i$ is a conjunction of the variables $x_j$ for
which there is an edge $j \rightarrow i$ in $\dd(f)$. These networks
were called \emph{conjunctive} networks in \cite{CBN}.

It is easy to see that the graph $\s(f)$ has the following structure.
Its connected components consist of a unique directed cycle, a \emph{limit cycle},
with directed
``trees"  feeding into the nodes of the cycle, representing the transient
states of the network, each of which eventually maps to a periodic point.
If the graph $\dd(f)$ is strongly connected, that is, there is a directed path
from any vertex to any other vertex, then it is shown in \cite{CBN} that there is a precise closed formula
for the number of limit cycles of any given length. This formula depends
on a numerical invariant of $\dd(f)$, its loop number. If $\s(f)$ is not strongly connected, then
there is a sharp lower bound and an upper bound on the number of limit cycles,
in terms of the loop numbers of the strongly connected components of $\s(f)$ and
the antichains in the poset of strongly connected components. These two bounds agree
for the number of fixed points of $f$, so that one corollary is a closed formula for the number
of fixed points of a general conjunctive network.

The current paper shows that these results hold in much broader generality, namely for time discrete
dynamical systems over any finite set $X$:
$$
f = (f_1,\ldots , f_n): X^n\longrightarrow X^n,
$$
such that the $f_i$ are constructed from an operator
$\ju:X^2\rightarrow X$, which has the property that it endows $X$
with the structure of a semilattice, that is, $\ju$ is commutative,
associative, and idempotent. We will call such functions $f_i$
\emph{semilattice functions}. That is, any semilattice $X$ gives
rise to a dynamical system to which the formulas and bounds for
limit cycles hold. It is shown furthermore that they hold precisely
for the class of semilattice networks.  It is important to mention
that since semilattice networks are not linear (see Example
\ref{eg:linear}), our results complement those for linear systems in
\cite{El,linear}.

\section{Semilattice networks}

Let $X$ be a set with $M$ elements and consider a dynamical system
with $n$ variables over $X$:
$$
 f = (f_1, \ldots , f_n): X^n\longrightarrow X^n.
$$

\begin{definition}\label{def:junc}
A function $\ju:X^2\rightarrow X$ is called a \emph{semilattice operator} if it
satisfies the following (with notation $x\ju y:=\ju(x,y)$):
\begin{eqnarray*}
  x\ju y &=& y\ju x  \textrm{ , $\ju$ is commutative} \\
  x\ju (y\ju z) &=& (x\ju y)\ju z \textrm{ , $\ju$ is associative}\\
  x\ju x &=& x \textrm{ , $\ju$ is idempotent}
\end{eqnarray*}
\end{definition}

Note that semilattice operators induce the structure of a semilattice on $X$. Conversely,
the meet or join operations on a semilattice are semilattice operators in the above sense.

\begin{example}
Some examples of semilattice operators are:
\begin{eqnarray*}
  \ju & = & AND  \textrm{ over } [0,1]^2 \\
  \ju & = & OR \textrm{ over } [0,1]^2 \\
  \ju & = & MIN \textrm{ over } [0,m]^2 \\
  \ju & = & MAX \textrm{ over } [0,m]^2 
\end{eqnarray*}
That is, semilattice operators are a generalization of the conjunctive and
disjunctive operators used in \cite{CBN}. \label{example:junctive}
\end{example}

The domain of a semilattice operator $\ju:X^2\longrightarrow X$ can
be extended to $\ju:X^k\longrightarrow X$ with $k\geq 1$ by
\begin{eqnarray*}
  x_1\ju x_2\ju \ldots \ju x_k &=& x_1\ju(x_2\ju(\ldots \ju(x_{k-1}\ju x_k))) \textrm{ if $k>1$ and} \\
  \ju x=x\ju x =x \\
\end{eqnarray*}
We will call such an extended operator a \emph{semilattice function}.

\begin{definition}
A dynamical system $f = (f_1, \ldots , f_n): X^n \longrightarrow
X^n$ is called a \emph{semilattice network} if there exists a semilattice operator,
$\ju$, such that $f_j=\Ju |_{X^k}$ is a semilattice function for all $j=1,\ldots,n$ (where
$f_j$ depends only on $x_{i_1},\ldots,x_{i_k}$).
\end{definition}

Let $f: X^n \longrightarrow X^n$ be a semilattice network, $G =
\dd(f)$, and $A$ the adjacency matrix of $G$. We will assume here and
in the remainder of the paper that none of the coordinate functions
of $f$ are constant, that is, all vertices of $G$ have positive
in-degree.

\subsection{Structure of the Dependency Graph} \label{adj-mat-sec}
Define the following
relation on the vertices of $G$: $a \sim b$ if and only if there is
a directed path from $a$ to $b$ and a directed path from $b$ to $a$.
It is easy to check that $\sim$ is an equivalence relation. Suppose
there are $t$ equivalence classes $V_1,\dots, V_t$. For each
equivalence class $V_i$, the subgraph $G_i$ with the vertex set
$V_i$ is called a \emph{strongly connected component} of $G$.
The graph $G$ is called \emph{strongly connected} if $G$ consists of a
unique strongly connected component. A trivial strongly connected
component is a graph on one vertex and no self loop. Since such
components do not influence the cycle structure of the network,
we assume that all strongly connected components are non-trivial.

\begin{example}\label{eg:junctivenetwork}
Our running example for this paper will be the semilattice network
$f:\{0,1,2\}^{63}\rightarrow \{0,1,2\}^{63}$ with dependency graph
given by Figure \ref{fig:example} and semilattice function
$\wedge=MIN$.
\end{example}

\begin{figure}[here]
\centerline{\framebox{\includegraphics[totalheight=8cm]{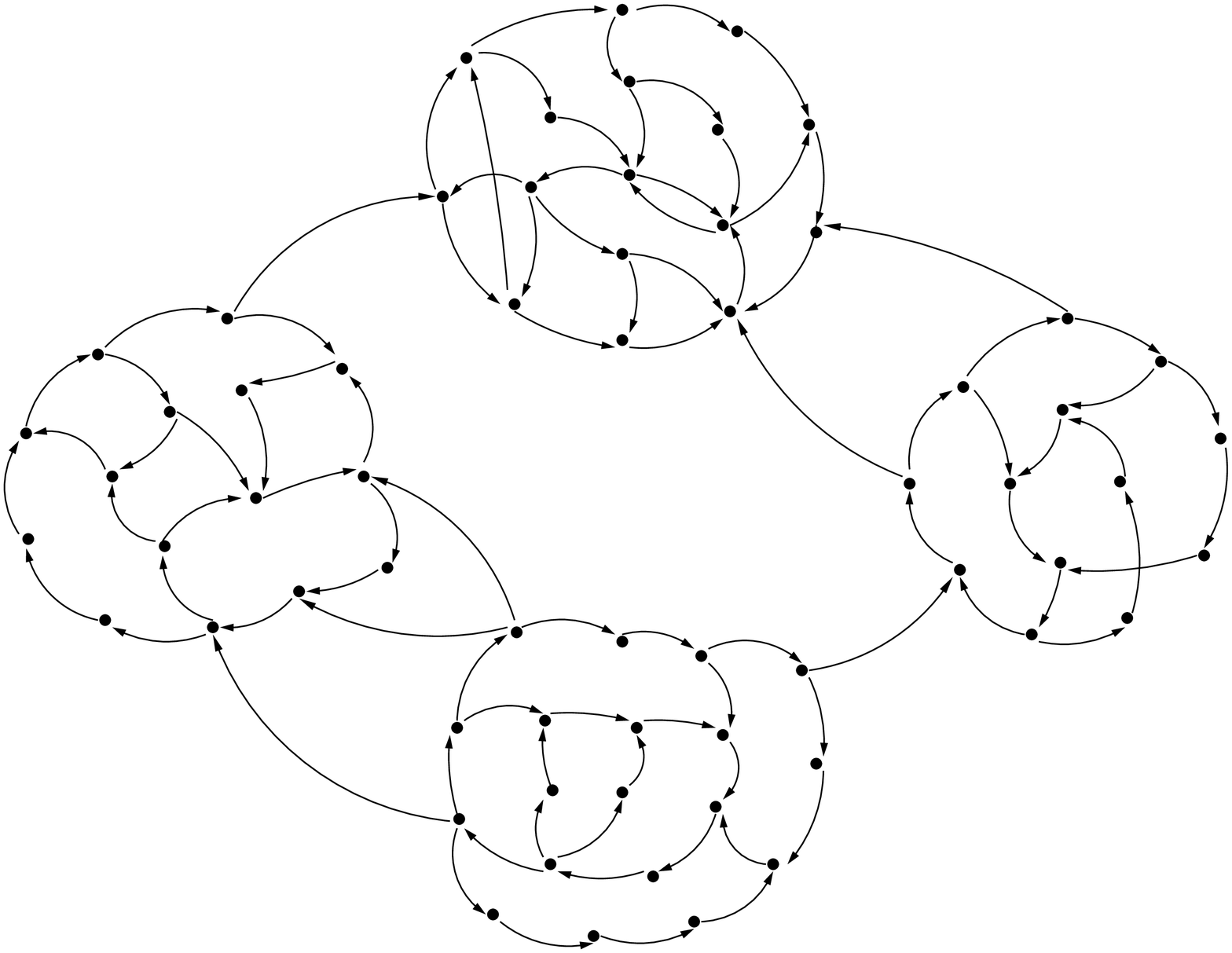}}}
\caption{Dependency graph of the semilattice network in Example
\ref{eg:junctivenetwork}. The labels have been omitted for
simplicity.} \label{fig:example}
\end{figure}

\vspace{.2cm} \noindent Let $G_i$ be a strongly connected component,
and let $h_i$ be the semilattice network with dependency graph
$\dd(h_i) = G_i$. Let $h: X^n \longrightarrow X^n$ be the semilattice
network defined by $h = (h_1,\dots,h_t)$. That is, the
dependency graph of $h$ is the disjoint union of the strongly
connected graphs $G_1,\ldots ,G_t$, and $h$ is obtained from
$g$ by deleting all edges between strongly connected components.

\noindent
Now define the following order relation on the strongly connected
components $G_1,\dots,G_t$ of the dependency graph $\dd(f)$ of the network $f$.
\begin{equation}
G_i \preceq G_j \mbox{ if there is at least one edge from a vertex in } G_i \mbox{ to a vertex in } G_j.
\end{equation}
In this way we obtain a partially ordered set $\mathcal{P}$. In this paper,
we relate the dynamics of $f$ to the dynamics of its strongly
connected components and the poset $\mathcal{P}$.

Example \ref{fig:example} (Cont.). The dependency graph of $f$ has
four strongly connected components, $G_1$, $G_2$, $G_3$, and $G_4$
(bottom, left, right and top, resp.). The poset is given by
$G_1\preceq G_2,G_3 \preceq G_4$.

\begin{figure}[here]
\centerline{\framebox{
\includegraphics[totalheight=3cm]{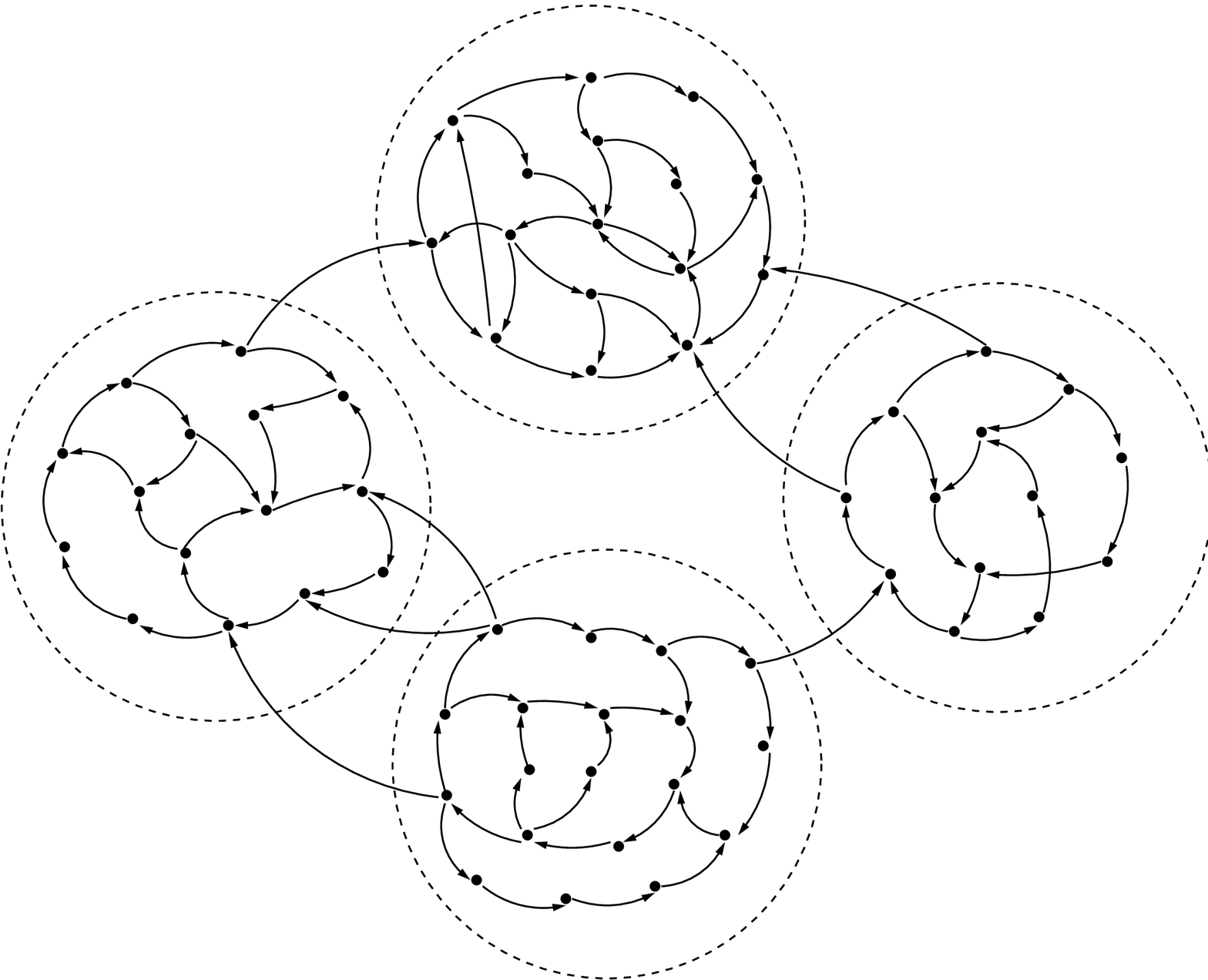}
\includegraphics[totalheight=3cm]{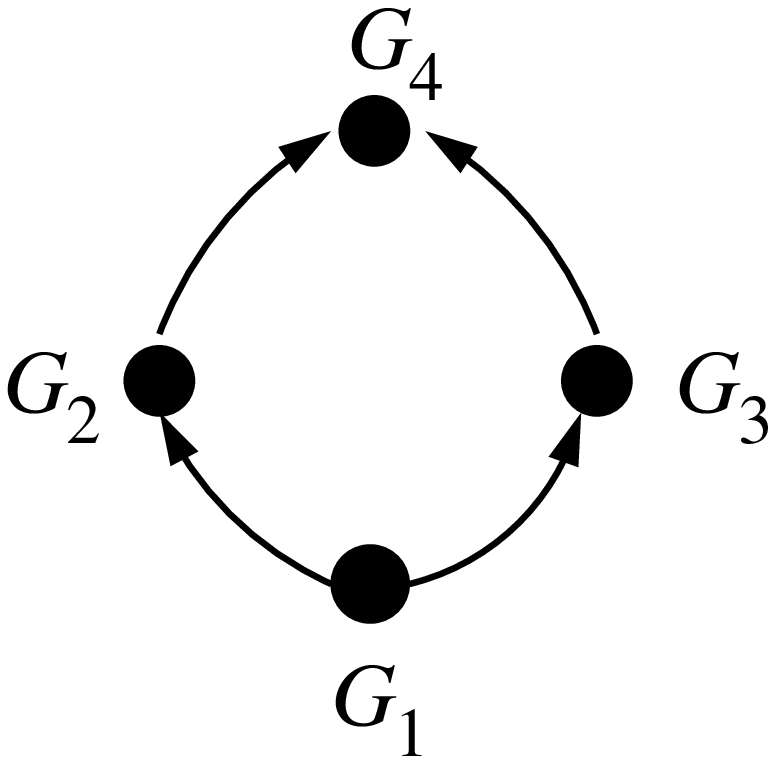}
}} \caption{The strongly connected components of $f$ (left) and
their poset (right).} \label{fig:examplecdg}
\end{figure}

\subsection{The Loop Number}

\begin{definition}
The \emph{loop number} of a strongly connected graph is the greatest
common divisor of the lengths of its simple (no repeated vertices)
directed cycles. The loop number of any directed graph $G$ is the
least common multiple of the loop numbers of its non-trivial
strongly connected components.
\end{definition}

Example \ref{fig:example} (Cont.). In our running example, the loop
numbers are $loop(G_1)=1$, $loop(G_2)=2$, $loop(G_3)=3$ and
$loop(G_4)=1$.

\section{Semilattice Networks with Strongly Connected Dependency Graph}
\label{stronglyconnected}

In this section we give an exact formula for the cycle structure of
semilattice networks with strongly connected dependency graphs. In the
next section we will also consider networks with general dependency
graphs and give upper and lower bounds for the cycle structure.

The formula for conjunctive networks was given and proven in
\cite{CBN}. It is not difficult to show that the proofs are also valid
for general semilattice networks.

\vspace{.2cm}\noindent Let $f:X^n \longrightarrow X^n$ be a semilattice
network with semilattice operator $\ju$. Assume that the dependency
graph $\dd(f)$ of $f$ is strongly connected with loop number $c$.

\begin{lemma}\label{lem:partition}
The set of vertices of $\dd(f)$ can be partitioned into $c$
non-empty sets $W_1,\dots,W_c$ such that each edge of $\dd(f)$ is an
edge from a vertex in $W_i$ to a vertex in $W_{i+1}$ for some $i$
with $1 \leq i \leq c$ and $W_{c+1} = W_1$.
\end{lemma}
\begin{proof}
For a proof of this fact see \cite[Lemma 3.4.1(iii)]{CombMath}
or \cite[Lemma 4.7]{CLP}.
\end{proof}
Let $s_i$ be the number of elements of $W_i$. Without loss of
generality we assume for the rest of the section that
$W_1=\{1,\ldots,s_1\}$, $W_2=\{s_1+1,\ldots,s_1+s_2\}$,\ldots,
$W_c=\{s_1+\ldots+s_{n-1}+1,\ldots,s_1+\ldots+s_c =n\}$,

\begin{proposition}
Let $c$ be the loop number of $f$; then there exists $k_0$ such
that for all $k\geq k_0$ we have
\begin{eqnarray*}
  f^{ck}(x)&=&(\underbrace{y_1,\dots,y_1}_{s_1 \
times}, \ldots,\underbrace{y_c,\dots,y_c}_{s_c \ times}) \\
  f^{ck+1}(x)&=&(\underbrace{y_c,\dots,y_c}_{s_1 \
times}, \underbrace{y_1,\dots,y_1}_{s_2 \ times},\dots,
\underbrace{y_{c-1},\dots,y_{c-1}}_{s_c \ times}) \\
&\vdots\\
  f^{ck+j}(x)&=&(\underbrace{y_{c+1-j},\dots,y_{c+1-j}}_{s_1 \
times}, \ldots,\underbrace{y_c,\dots,y_c}_{s_j \ times},
\underbrace{y_1,\dots,y_1}_{s_{j+1} \ times},\dots,
\underbrace{y_{c-j},\dots,y_{c-j}}_{s_c \ times})\\
&\vdots\\
  f^{ck+c}(x)&=&f^{ck}(x)
\end{eqnarray*}
where $y_j=\Ju_{i\in W_j}x_i$ \label{prop:iterations}
\end{proposition}
\begin{proof}
The proof is analogous to \cite[Theorem 4.10]{CLP}.
\end{proof}

The following corollary states that the period of $f$ can be
obtained from the topology of its dependency graph.
\begin{corollary}\label{cor:period}
The period of $f$ is equal to the loop number of $\dd(f)$.
\end{corollary}

The following corollary states that the long-term dynamics of a
semilattice network can be reduced to the dynamics of a rotation over
$X^c$, with $c$ as in Lemma \ref{lem:partition}.

\begin{corollary}\label{cor:rotation}
The cycle structure of $f$ is equal to the cycle structure of
$R:X^c\rightarrow X^c$ where
$R(y_1,\ldots,y_c)=(y_c,y_1,\ldots,y_{c-1})$.
\end{corollary}
\begin{proof}
Let $\Gamma:X^c \rightarrow X^n$ and $\Phi:X^n \rightarrow X^c$ be
defined by
$\Gamma(y_1,y_2,\ldots,y_c)=(\underbrace{y_1,\dots,y_1}_{s_1 \
times}, \underbrace{y_2,\dots,y_2}_{s_2 \ times},\dots,
\underbrace{y_c,\dots,y_c}_{s_c \ times}) $ and $
\Phi(x_1,\dots,x_n)= (x_{s_1},x_{s_1+s_2},\ldots,x_n)$. The proof
now follows from the equalities $\Phi\circ f\circ\Gamma=\tau$ and
$\Gamma\circ\tau\circ\Phi=f|_{\{\textrm{periodic points of $f$}\}}$.
\end{proof}

For any positive integers $p,k$ that divide $c$, let $A(p)$ be the
set of periodic points of period $p$ and let $\ds D(k) :=
\bigcup_{p|k} A(p)$.
\begin{proposition}
 The cardinality of the set $D(k)$ is $|D(k)|=M^k$.
\end{proposition}
\begin{proof}
It follows from the fact that if $k|c$, then a rotation in $M$
colors and $c$ variables has $M^k$ colorings of periods that divide
$k$.
\end{proof}

\begin{corollary}
If $p$ is a prime number and $p^k$ divides $c$ for some $k \geq 1$, then
\begin{displaymath}
|A(p^k)| = M^{p^k} - M^{p^{k-1}}.
\end{displaymath}
\end{corollary}

\begin{proof}
It is clear that $|D(1)| = M$ (there are $M$ constant colorings).
Now if $p$ is prime and $k \geq 1$, then the proof follows from the
fact that $D(p^k) = D(p^{k-1}) \biguplus A(p^k)$, where $\biguplus$
is the disjoint union.
\end{proof}

Next we prove Theorem \ref{am-formula} which gives the exact number of periodic points of
any possible length.

\begin{theorem}\label{am-formula}
Let $f$ be a semilattice network whose dependency graph is
strongly connected and has loop number $c$.  If $c=1$, then
$\cc(f)=MC^1$. If $c>1$ and $k=p_1^{s_1}\ldots p_r^{s_r}$ is a
divisor of $c$, then the number of periodic states of period $k$ is
$|A(k)|$
\begin{equation} |A(k)|= \sum_{i_1 = 0}^1 \cdots \sum_{i_r=0}^{1}
(-1)^{i_1+i_2+\cdots+i_r} M^{p_1^{s_1-i_1}p_2^{s_2-i_2} \dots
p_r^{s_r-i_r}}. \label{Am}
\end{equation}
\end{theorem}
\begin{proof}
The statement for $c=1$ is part of the previous corollary. Now
suppose that $c>1$. For $1 \leq j \leq r$, let $k_j =
p_j^{s_j-1}\prod_{i=1, i\neq j}^r p_i^{s_i}$. Then $D(k) = A(k)
\biguplus (\bigcup_{j=1}^r D(k_j)) $, where $\biguplus$ is a
disjoint union, in particular,
\begin{displaymath}
A(k) = D(k) \setminus \bigcup_{j=1}^r  D(k_j).
\end{displaymath}
The formula \ref{Am} follows from the inclusion-exclusion principle and the disjoint union above.
\end{proof}

\begin{corollary}
If $k$ divides $c$, then the number of cycles of length $k$ in the
phase space of $f$ is $\cc(f)_k = \dfrac{|A(k)|}{k}$. Hence the
cycle structure of $f$ is
\begin{equation*}
\cc(f) = \sum_{k \mid c} \dfrac{|A(k)|}{k} \cc^k.
\end{equation*}
\end{corollary}

Example \ref{fig:example} (Cont.). In our running example we obtain:
\[
\cc(h_1)= 3\cc_1, \ \ \ \ \cc(h_2)= 3\cc_1+3\cc_2, \ \ \ \ \cc(h_3)=
3\cc_1+8\cc_3, \ \ \ \  \cc(h_4)= 3\cc_1
\]

\begin{remark}
Notice that the cycle structure of $f$ depends on the loop number
and $|X|=M$ only. In particular, a semilattice network with
loop number 1 on a strongly connected dependency graph only has as
limit cycles the $M$ fixed points $(w,w,\ldots ,w)$ where $w\in X$,
regardless of how many vertices its dependency graph has and how
large $n$ is.
\end{remark}

\section{Networks with general dependency graph}

Let $f : \s \longrightarrow \s$ be a semilattice network with
dependency graph $\dd(f)$. 
Let $G_1,\dots,G_t$ be the strongly connected components of
$\dd(f)$. 
Furthermore, suppose that none of the $G_i$ is trivial.
 For $1 \leq i \leq t$, let $h_i$ be the semilattice network
that has $G_i$ as its dependency graph and suppose that the loop
number of $h_i$ is $c_i$. In particular, the loop number of $f$ is
$c := \lcm\{c_1,\dots,c_t\}$.

First we study the effect of deleting an edge in the dependency
graph between two strongly connected components. Let $G_1$ and $G_2$
be two strongly connected components in $\dd(f)$ and suppose $G_1
\preceq G_2$. Let $x \longrightarrow y$ be a directed edge in
$\dd(f)$ between a vertex $x$ in $G_1$ and a vertex $y$ in $G_2$.
Let $\dd'$ be the graph $\dd(f)$ after deleting this edge, and let
$g$ be the semilattice network such that $\dd(g) = \dd'$.

\begin{theorem}\label{theorem:f-h}
Any cycle in the phase space of $f$ is a cycle in the phase space of
$g$. In particular  $\cc(f) \leq \cc(g)$ componentwise.
\end{theorem}
\begin{proof}
First notice that $x$ appears in the expression $f^k_y$ if and only
if there is a path from $x$ to $y$ of length $k$.
Let $\mathcal{C} := \{\uu,f(\uu),\dots,f^{m-1}(\uu)\}$ be a cycle of
length $m$ in $\s(f)$. To show that $\mathcal{C}$ is a cycle in
$\s(g)$, it is enough to show that, $\uu_x\ju \uu_{G_2}=\ju
\uu_{G_2}$. Thus, the value of $y$ is determined already by the
value of $y'$ and the edge $(x,y)$ does not contribute anything new
and hence $\cc$ is a cycle in $\s(g)$. This is equivalent to show
that $\uu_x\ju \uu_{y'}=\uu_{y'}$ for all $y' \in G_2$ such that
there is an edge from $y'$ to $y$.

Suppose the loop number of $G_1$ (resp. $G_2$) is $a$ (resp. $b$).
Now, any path from $x$ (resp. $y$) to itself is of length $pa$
(resp. $qb$) where $q, p \geq T$ and $T$ is large enough, see
\cite[Corollary 4.4]{CLP}. Thus there is a path from $x$ to $y$ of
length $qa+1$ for any $q \geq T$. Also, there is a directed path
from $y$ to $y'$ of length $qb-1$ for any $q \geq T$. This implies
the existence of a path from $x$ to $y'$ of length $q(a+b)$ for all
$q \geq T$, then $x$ appears in $f^{mk}_{y'}$.

Now $\uu = f^{m}(\uu) = f^{mk}(\uu)$, for all $k \geq 1$. Choose
$q,k$ large enough such that $q(a+b) = km \geq T$. Then,
$\uu_{y'}=f^{mk}_{y'}(\uu)=\uu_x\ju\ldots$ and $\uu_x\ju
\uu_{y'}=\uu_x\ju \uu_x\ju\ldots=\uu_x\ju\ldots=\uu_{y'}$. Therefore
$\uu_x\ju \uu_{G_2}=\ju \uu_{G_2}$.
\end{proof}

Let $h :\s \longrightarrow \s$ be the semilattice network with the
disjoint union of $G_1, \dots, G_t$ as its dependency graph. That
is, $h = (h_1,\dots,h_t)$. For $h$, there are no edges between any
two strongly connected components of the dependency graph of the
network. Its cycle structure can be completely determined from the
cycles structures of the $h_i$ alone.

\begin{theorem}\label{disjoint}
Let $\cc(h_i) = \sum_{j|l_i} a_{i,j} \cc^j$ be the cycle structure
of $h_i$. Then the cycle structure of $h$ is $\cc(h)= \prod_{i=1}^t
\cc(h_i)$ (where $\cc^r \cc^s := \frac{rs}{\lcm(r,s)}
\cc^{\lcm(r,s)}$) and the number of cycles of length $m$ (where
$m|l$) in the phase space of $h$ is
\begin{equation}\label{C_m:disjoint}
\cc(h)_m = \sum_{\substack{j_i | l_i \\ \lcm\{j_1,\dots,j_t\}=m}}
\dfrac{j_1 \cdots j_t}{m} \prod_{i=1}^t a_{i,j_i}.
\end{equation}
\end{theorem}
\begin{proof}
This follows from the fact that if $\uu$ is a periodic point of
$h_i$ of period $k_i$ and  $\vv$ is a periodic point of $h_j$ of
period $k_j$, then $(\uu,\vv)$ is a periodic point of $(h_i,h_j)$ of
period $\lcm(k_i,k_j)$.
\end{proof}

\begin{corollary}\label{upperbound}
Let $f$ and $h$ be as above. The number of cycles of any length in
the phase space of $f$ is less than or equal to the number of cycles
of that length in the phase space of $h$. That is $\cc(f) \leq
\cc(h)$ componentwise. In particular, the period of $f$ is a divisor
of the loop number of its dependency graph.
\end{corollary}


In \cite{CBN}, it was shown that the poset structure of
$\mathcal{P}$ gives an algebraic way to combine the cycle structure
of $h_i$ to obtain lower and upper bounds for the cycle structure of
$f$, where $f$ was a conjunctive Boolean network. It is not
difficult to see that the proofs of these results still also hold
for general semilattice networks if the corresponding semilattice
operator has a ``neutral'' and an ``absorbent'' element (analogous
to the identities $1\wedge x=x$ and $0\wedge x=0$). In order to
state the theorem about lower and upper bounds of semilattice
networks, we need the following definitions.

\begin{definition}
Let $\ju:X^2\rightarrow X$ be a semilattice operator. An element
$\la\in X$ is called a \emph{neutral element} if $\la\ju x=\ju x$
for all $x\in X$. An element $\te\in X$ is called an \emph{absorbent
element} if $x\ju \te=\ju \te$ for all $x\in X$.
\end{definition}

\begin{example}
All the functions in Example \ref{example:junctive} have a neutral
and absorbent element; they are:
\begin{eqnarray*}
  \la = 1& , & \te=0 \\
  \la = 0& , & \te=1  \\
  \la = m& , & \te=0 \\
  \la = 0& , & \te=m\\
\end{eqnarray*}
\end{example}

\begin{remark}
Since $X$ is finite, every semilattice operator has an absorbent
element ($\te=\ju_{x\in X}x$).
\end{remark}

\begin{remark}
Any semilattice operator $\ju$ on $X$ can be extended to a set with a
neutral element by defining $\ju':(X\cup \{\la\})^2\rightarrow X\cup
\{\la\}$ as $x \ju' y=x\ju y$ if $x,y\in X$ and $x \ju' \la =\la
\ju' x=x$ otherwise.
\end{remark}

Let $f$, $h$, $G_1,\dots,G_t$ be as above and let $c_i$ be the loop
number of $G_i$. Furthermore, let $\mathcal{P}$ be the poset of the
strongly connected components. Let $\Omega$ be the set of all
maximal antichains in $\mathcal{P}$. For $J \subseteq [t] =
\{1,\ldots ,t\}$, let $x_J := \prod_{j \in J} x_j$ and let
$\overline{J} := [t]\setminus J$. In the remainder of the paper we
assume that $f:X^n\rightarrow X^n$ is a semilattice network such that
its semilattice operator, $\ju$, has idempotent neutral and absorbent
elements, $\la$ and $\te$, resp.

\begin{definition}
For any subset $J \subseteq [t]$, let
\[
J^{\preceq} := \{k \, : \,  G_j \preceq G_k \mbox{ for some } j \in J\} , \,
J^{\succeq} := \{k \, : \,  G_j \succeq G_k \mbox{ for some } j \in J\},
\]
\[
J^{\prec} := \{k \, : \,G_j \prec G_k \mbox{ for some } j \in J\} \mbox{, and }
J^{\succ} := \{k \, : \,G_j \succ G_k \mbox{ for some } j \in J\}.
\]
A limit cycle $\cc$ in the phase space of $f$ is $J_\te$ (resp.
$J_\la$) if the $G_j$ component of $\cc$ is $\te$ (resp. $\la$) for
all $j \in J$.
\end{definition}

Denote $\langle K,L\rangle=
\begin{cases}
0, & \text{if $K \cap L \neq \emptyset$},\\
1, & \text{if $K \cap L = \emptyset $}.
\end{cases}$ and $I_N=I^{\succeq}\cup N,J^M=J^{\preceq}\cup M $.

\begin{definition}
The $L$- and $U$-polynomial associated to $\mathcal{P}$ are defined
as follows:
\begin{eqnarray*}
  \mathcal{L}(z_1,\dots,z_t) &=& \sum_{\jj \subseteq
\Omega} (-1)^{|\jj|+1} \prod_{k\in \bigcap_{J \in \jj} J} z_k \\
  \mathcal{U}(z_1,\dots,z_t) &=& \sum_{\substack{I\subseteq N\subseteq [t] \\
J\subseteq M\subseteq [t]}} {(-1)^{|N|+|M|+|I|+|J|}
 {\langle I_N,J^M\rangle \prod_{k\in
\overline{I_N\cup J^M}}{z_k} }}
\end{eqnarray*}
\end{definition}

Example \ref{fig:example} (Cont.). For the poset in Figure
\ref{fig:examplecdg} we obtain:
\begin{align*}
\mathcal{L}(z_1,z_2,z_3,z_4) &= -2+z_1+z_2z_3+z_4\\
\mathcal{U}(z_1,z_2,z_3,z_4) &= 14-7z_1+3z_1z_2-4z_2+3z_1z_3-z_1z_2
z_3+z_2z_3-4z_3+4z_1z_4-2z_1z_2z_4\\
& +3z_2z_4-2z_1z_3z_4+z_1 z_2 z_3 z_4-z_2 z_3 z_4+3 z_3 z_4-7 z_4
\end{align*}

\begin{theorem}\label{thm:main}
With the notation above we have the following coefficient-wise
inequalities
\[
\mathcal{L}(\cc(h_1),\dots,\cc(h_t)) \leq \cc(f) \leq
\mathcal{U}(\cc(h_1),\dots,\cc(h_t)). \label{ineq}
\]
Here, the polynomials $\mathcal{L}$ and $\mathcal{U}$ are evaluated
using the ``multiplication" described in Theorem \ref{disjoint} and
coefficient-wise addition.
\end{theorem}

\begin{proof}
The proof is analogous to the proofs of Theorems 6.2 and 7.4 in \cite{CBN}.
\end{proof}

Note that the left and right sides of the inequalities (\ref{ineq})
are polynomial functions in the variables $\cc(h_i)$, with integer
coefficients. That is, the lower and upper bounds are polynomial
functions depending exclusively on measures of the network topology.

Example \ref{fig:example} (Cont.). In our running example we obtain:
\begin{align*}
\mathcal{L}(\cc(h_1),\cc(h_2),\cc(h_3),\cc(h_4)) &= 13\cc_1+9\cc_2+24\cc_3+24\cc_6 \\
\mathcal{U}(\cc(h_1),\cc(h_2),\cc(h_3),\cc(h_4)) &=
20\cc_1+24\cc_2+64\cc_3+96\cc_6
\end{align*}

and therefore $13\cc_1+9\cc_2+24\cc_3+24\cc_6 \leq \cc(f) \leq
20\cc_1+24\cc_2+64\cc_3+96\cc_6$. It is important to mention that
the phase space of $f$ has $3^{63}\approx 10^{30}$ nodes so it is
not feasible to obtain information about the cycle structure from
exhaustive enumeration. Also, although the bounds agree on fixed
points for Boolean semilattice networks \cite{CBN}, our example
shows that they do not agree for general lattice networks.

\section{Characterization of Semilattice Networks}

In this section we characterize semilattice networks; in order to do
this, it is enough to characterize semilattice operators. Since
semilattice operators are semilattice operations, the number of
semilattice operators over a set with $m$ elements (up to permutation)
is the number of semilattices with $m$ elements. According to the
next proposition, the number of semilattice operators over a set with
$m$  elements is the number of lattices of size $m+1$. Although
there is no closed formula for the number of lattices of a given size, algorithms
for counting them have been developed \cite{CFL}.

\begin{proposition}
There is a one-to-one correspondence between semilattices with $m$
elements and lattices with $m+1$ elements.
\end{proposition}
\begin{proof}
If $(X,\ju)$ is a semilattice with $m$ elements, consider $(X\cup
\{\la\},\ju',\vee)$, by defining $x \ju' y=x\ju y$ if $x,y\in X$ and
$x\ju' \la =\la \ju' x=x$ otherwise; also, $x \vee y=\ju \{z:z\ju
x=x,z\ju y=y\}$. It follows that $(X\cup \{\la\},\ju',\vee)$ is a
lattice with $m+1$ elements. On the other hand, if $(Z,\ju,\vee)$ is
a lattice with $m+1$ elements, let $\la=\Ju Z$; then it follows that
$(Z\backslash \{\la\},\ju|_{Z\backslash \{\la\}^2})$ is a
semilattice with $m$ elements.
\end{proof}

\section{Infinite Semilattice Networks}
We present some results on infinite semilattice networks, both
networks on an infinite set $X$ and networks on an infinite
Cartesian product of a set $X$.

\subsection{Semilattice networks on infinite sets}
If $X$ is a set with infinitely many elements (that is, $M=\infty$),
some of the theorems remain valid. Notice that if $S\subseteq X$ is
finite, there exists a finite set $Z\supseteq S$ that is closed
under $\ju$; that is, $\ju|_Z:Z^2\rightarrow Z$ is a semilattice
function.

Suppose that $\dd(f)$ is strongly connected with loop number $c$,
and let $\uu=(u_1,\ldots,u_n)$ be a periodic point of $f$. Let
$Z\supseteq \{u_1,\ldots,u_n\}$ be a finite subset of $X$ that it is closed under $\ju$;
then we can consider $\uu$ as a periodic point of $f|_{Z^n}$.
Therefore, the results on finite semilattice networks apply, and the period of $\uu$ must divide
$c$. That is, Corollary \ref{cor:period} is valid for $M=\infty$.
Now, consider a divisor $k$ of $c$ and consider $Z\subseteq X$ finite with at least
$m$ elements such that $\ju|_{Z^2}$ is a semilattice operator. Then,
the number of periodic states of period $k$ of $f|_{Z^n}$ is at
least $|A(k)|$ (see Theorem \ref{am-formula} and notice that
$|A(k)|$ is increasing with respect to $M$). Since
$\lim_{m\rightarrow \infty}|A(k)|=\infty$, it follows that $f$ has
infinitely many periodic points and limit cycles of length $k$.
Then Theorem \ref{am-formula} (and the corresponding corollary)
holds for $M=\infty$.

If $\dd(f)$ is not necessarily strongly connected, suppose $\ju$ has
a neutral and an absorbent element. Let $h_1,\ldots,h_n$ correspond to
$f|_{Z^n}$Then, by using the argument in
the paragraph above, it follows that if at least a limit cycle of
length $k$ appears in $\mathcal{L}(\cc(h_1),\ldots,\cc(h_t))$, then $f$ has infinitely
many limit cycles of length $k$.

\subsection{Infinite-dimensional semilattice networks}
If the dimension of $f$ is infinite, that is, $n=\infty$, then $\ju$
needs to satisfy an additional condition to be properly defined: Every
(possibly infinite) subset $S\subseteq X$ has a largest lower bound.
This means that there exists $x\in X$ such that $x\ju s=x$ for all $s\in
S$ (in lattice terminology: $x$ is a lower bound); and if there is
another such $x'$, then $x\ju x'=x'$ (in lattice terminology: $x$ is
the largest lower bound). This additional condition allows for a
function $f_i$ to have infinitely many inputs.

Suppose that $\dd(f)$ is strongly connected with loop number $c$,
and let $\uu=(u_1,\ldots)$ be a periodic point of $f$ of period $d$.
Consider $x,y\in W_i$ (see Lemma \ref{lem:partition}); then there
exists $k_0$ such that for all $k\geq k_0$ there is a path of length
$ck$ from $x$ to $y$ and from $y$ to $x$. Then,
$f^{ck}_y=x\ju\ldots$, $f^{ck}_x=y\ju\ldots$ for all $k\geq k_0$; in
particular, $u_y=f^{ck_0d}_y(\uu)=u_x\ju w$ and
$u_x=f^{ck_0d}_x(\uu)=u_y\ju v$ for some $v,w\in X$. Then, $u_x\ju
u_y=u_x\ju u_x\ju w=u_x\ju w=u_y$, similarly $u_x\ju u_y=u_x$;
therefore $u_x=u_y$. It follows that Corollary \ref{cor:period} and
Theorem \ref{am-formula} remain valid for $n=\infty$ (and $M\in
\mathbb{Z}^+\cup \{\infty\}$).

If $\dd(f)$ is not necessarily strongly connected, suppose $\ju$ has
a neutral and an absorbent element. It is not difficult to show that
Theorem \ref{thm:main} remains valid for $n=\infty$ (and $M\in
\mathbb{Z}^+\cup \{\infty\}$).

\bigskip\noindent
Finally, we show with counterexamples that we cannot omit any
of the properties of $\ju$ in Definition \ref{def:junc}. That is, the
formulas and bounds derived in this paper are valid exactly for
semilattice networks.

\begin{example}\label{eg:linear}
Consider $\ju=+$, then $\ju$ is commutative and associative, but not
idempotent. Consider $f:\mathbb{F}_2^2\rightarrow \mathbb{F}_2^2$,
defined by $f(x_1,x_2)=(x_1+x_2,x_1+x_2)$. It is not difficult to
see that $f$ has a unique limit cycle (a fixed point),
$(x_1,x_2)=(0,0)$; that is, $\cc(f)=\cc^1$. On the other hand, the
loop number of its dependency graph is $1$, so from Theorem
\ref{am-formula} we would obtain $2\cc^1\neq \cc(f)$.
\end{example}

\begin{example}
Consider $x \ju y=x^2y$ defined over $\mathbb{F}_3$. It is easy to
show that $\ju$ associative and idempotent, but not commutative.
Consider $f:\mathbb{F}_3^2\rightarrow \mathbb{F}_3^2$, defined by
$f(x_1,x_2)=(x_1^2 x_2,x_1)$. It is not difficult to show that $f$
has the limit cycle of length 2, $\{(1,2),(2,1)\}$ (it also has 3
fixed points). On the other hand, the loop number of its dependency
graph is $1$, so from Theorem \ref{am-formula} we would obtain
$3\cc^1\neq \cc(f)$.
\end{example}

\begin{example}
Consider $x \ju y=2x+2y$ defined over $\mathbb{F}_3$. It is easy to
show that $\ju$ is commutative and idempotent, but not associative.
Consider $f:\mathbb{F}_3^2\rightarrow \mathbb{F}_3^2$, defined by
$f(x_1,x_2)=(2x_1+2x_2,x_1)$. It is not difficult to show that $f$
has the limit cycle of length 3, $\{(0,1),(2,0),(1,2)\}$ (it also
has 3 fixed points and another limit cycle of length 3). On the
other hand, the loop number of its dependency graph is $1$, so from
Theorem \ref{am-formula} we would obtain $3\cc^1\neq \cc(f)$.
\end{example}

\section{Discussion}
In this paper we have identified a broad class of discrete dynamical
systems with a finite phase space for which one can derive
strong results about their long-term dynamics in terms of
properties of their dependency graphs. We classify completely
the limit cycles of semilattice networks with strongly connected
dependency graph and provide polynomial upper and lower bounds
in the general case. It is our hope that the formulas in this paper
are related to general properties of semilattices, which is a subject
for future investigation. As mentioned in the Introduction, the motivation
for this investigation was the need for theoretical tools to analyze
discrete models in biology. An example of such an application is
given in \cite{ABM}, where it is shown that the results about conjunctive
Boolean networks can be applied to determining the limiting behavior of
certain types of epidemiological models.

The results in this paper apply to certain types of Boolean networks and
cellular automata, which in many cases have the property that the
update functions are of the same type for all nodes. Another model type
to which the results in this paper apply in some cases is that of
so-called logical models, developed by Ren\'e Thomas for the purpose
of modeling gene regulatory networks. It is shown in \cite{VJL} that
logical models can be translated into the framework of polynomial dynamical
systems. If the dynamical system arises from a semilattice function, then
the results of this paper give information about the steady states
and limit cycles of the model under synchronous update.

\end{document}